\documentclass{amsart}

\usepackage[english]{babel}
\usepackage{amsfonts, amsmath, amsthm, amssymb,amscd,indentfirst}
\usepackage{graphicx}
\usepackage{epstopdf}
\usepackage{amsmath,amssymb,latexsym,indentfirst}
\usepackage{enumerate}
\usepackage{times}
\usepackage{palatino}

\usepackage{tikz}
\usetikzlibrary{calc}

\newtheorem{theorem}{Theorem}[section]

\newtheorem{proposition}{Proposition}[section]
\newtheorem{lemma}{Lemma}[section]

\newtheorem{corollary}{Corollary}[section]
\newtheorem{observation}{Observation}[section]

\newtheorem{example}{Example}[section]

\def\R{\mathbb{R}}

\def\bp{\begin{proof}}
\def\ep{\end{proof}}

\def\R{{\cal R}}

\def\R{\mathbb{R}}

\def\B{\mathcal{B}}

\def\R{\mathbb{R}}

\def\H{\mathbb{H}}         
\def\S{\mathbb{S}}         
\def\B{\mathbb{B}}

\begin{document}

\title[Gap results for free boundary CMC surfaces]
{Gap results for free boundary CMC surfaces in conformally Euclidean three-balls}

\author{Maria Andrade}
\address{Universidade Federal de Sergipe, Departamento de Matem\'{a}tica, 49100-000, S\~ao Crist\'ov\~ao, SE, Brasil and Universidade Federal de Minas Gerais (UFMG), Departamento de Matem\'{a}tica, Caixa Postal 702, 30123-970, Belo Horizonte, MG, Brasil.}
\email{maria@mat.ufs.br}
\author{Ezequiel Barbosa}
\address{Universidade Federal de Minas Gerais (UFMG), Departamento de Matem\'{a}tica, Caixa Postal 702, 30123-970, Belo Horizonte, MG, Brazil}
\email{ezequiel@mat.ufmg.br}
\author{Edno Pereira}
\address{Universidade Federal de Minas Gerais (UFMG), Departamento de Matem\'{a}tica, Caixa Postal 702, 30123-970, Belo Horizonte, MG, Brazil}
\email{ednoalan@ufmg.br}

\thanks{The authors were partially supported by  CNPq, CAPES and FAPEMIG/Brazil agencies grants.}

\date{}

\begin{abstract} In this work, we consider
 $M=(\B^3_r,\bar{g})$ as the Euclidean three-ball with radius $r$ equipped with the metric $\bar{g}=e^{2h}\left\langle , \right\rangle$ conformal to the Euclidean metric. We show that if a free boundary CMC surface $\Sigma$ in $M$ satisfies a pinching condition on the length of the traceless second fundamental tensor  which involves the support function of $\Sigma$, the positional conformal vector field $\vec{x}$ and its potential function $\sigma,$ then either $\Sigma$ is a disk or $\Sigma$ is an annulus rotationally symmetric. In a particular case, we construct an example of minimal surface with strictly convex boundary in $M$, when $M$ is the Gaussian space, that illustrate our results. These results  extend to the CMC case and to many others different conformally Euclidean spaces the main result obtained by Haizhong Li and Changwei Xiong in \cite{HaXa}. 
\end{abstract}

\maketitle


\vspace{0.5cm}
\section{Introduction}\label{intro}

Let $M$ be a three-dimensional Riemannian manifold with smooth boundary $\partial M$. Let $x:\Sigma \to M$ be an isometric immersion, where $\Sigma$ is a smooth compact surface with $\partial \Sigma \subseteq \partial M.$ As well known, $\Sigma$ is a free boundary CMC surface, if the mean curvature is constant and $T\Sigma$ is orthogonal to $T\partial M$ at every point of $\partial \Sigma.$ The first variation formula shows that free boundary CMC surfaces are critical points of the area functional for volume preserving variations of $\Sigma,$ whose $\partial \Sigma$ is free to move in $\partial \Sigma.$ In the setting where a minimal surface lies in a three-dimensional Euclidean unit ball $\mathbb {B}^3$ with free boundary, the flat equatorial disk and the critical catenoid are the most well known examples. Here the critical catenoid is a piece of a catenoid in $\mathbb{R}^3$ which intersects $\partial \mathbb{B}^3$ orthogonally. In the setting where a CMC surface, with non-zero mean curvature, lies in a three-dimensional Euclidean unit ball $\mathbb {B}^3$ with free boundary, the spherical caps and the pieces of Delaunays surfaces, inside the ball which intersects $\partial \mathbb{B}^3$ orthogonally, are the most well known examples. Many abstract examples of free boundary minimal surfaces in the unit three-dimensional ball had been recently constructed by using the desingularization method or the gluing method (see, for instance, \cite{FS16, FPZ, K, KW}), and it is expected these methods can also be used to build  other examples of free boundary CMC surfaces with high genus and many boundary components. 

In this work, we are interested in the classification problem considering a pinching condition involving the length of the second fundamental form and the support function of the free boundary CMC surfaces. In 2016, Ambrozio and Nunes, see \cite{ambrozio2016gap}, provided the first characterization result on the equatorial disk and the critical catenoid in three-dimensional unit Euclidean ball using a condition involving the length of the second fundamental form and the support function. More precisely, they proved the following:

\begin{theorem}[\cite{ambrozio2016gap}] 
 Let $\Sigma$ be a compact free boundary minimal surface in $\mathbb{B}^3.$ Assume that for all points $x\in\Sigma,$
 $$|A|^2(x)\langle x,N(x)\rangle^2\leq 2,$$
 where $N(x)$ denotes a unit normal vector at the point $x\in \Sigma$ and $A$ denotes the second fundamental form of $\Sigma,$ then,
 \begin{enumerate}[i)]
  \item either $|A|^2(x)\langle x, N(x)\rangle^2=0$ and $\Sigma$ is a flat equatorial disk,
  \item or $|A|^2(p)\langle p, N\rangle ^2=2$ at some point $p$ and $\Sigma$ is a critical catenoid.
 \end{enumerate}
\end{theorem}

 In one direction, motivated by work's of Ambrozio and Nunes, the first and second named authors in a joint work  with Cavalcante (see \cite{barbosa2019gap}) proved analogous result of their in context of three boundary CMC surfaces in the Euclidean three-ball $\mathbb{B}^3.$
 
 \begin{theorem}[\cite{barbosa2019gap}] Let $\Sigma$ be a compact free boundary CMC surface in $\mathbb{B}^3.$ Assume that for all $x\in\Sigma,$
 \begin{eqnarray}\label{eq001}
 |\Phi|^2\langle x, N\rangle^2\leq \dfrac{1}{2}\left(2+H\langle x, N\rangle\right)^2,
 \end{eqnarray}
where $\Phi=II-\frac{H}{2}g_{\Sigma},$ $II(X,Y)=g_{\Sigma}(A(X),Y)$ is the tensor associate with the second fundamental form $A$ of $\Sigma.$ Then,
\begin{enumerate}[i)]
 \item either $|\Phi|^2\langle x, N\rangle ^2\equiv 0$ and $\Sigma$ is spherical cap,
 \item or equality in \eqref{eq001} occurs at some point and $\Sigma$ is a part of Delaunay surface.
\end{enumerate}

 \end{theorem}

 In other direction, H. Li and C. Xiong (see \cite{HaXa}), considered $M=[0,R_{\infty})\times\mathbb{S}^2$ as a $3$-dimensional warped product Riemannian manifold with the metric $g=dr^2+\lambda(r)^2g_{_{\mathbb{S}^2}},$ where $g_{_{\mathbb{S}^2}}$ is the canonical metric in unit sphere $\mathbb{S}^2$ and $\lambda(r)=r$ or $\lambda(r)=\sinh(r)$ if $R_{\infty}=+\infty$ or $\lambda(r)=\sin(r)$ when $R_{\infty}=\pi/2$, that is, $M$ is the Euclidean space $\mathbb{R}^3,$  the hyperbolic space $\mathbb{H}^3$ or hemisphere $\mathbb{S}_+^3$, respectively. About the  vector field $X=\lambda(r)\partial _r,$ we have two properties:  the first one is that Lie derivative of $g$ in  the direction $X$ satisfies 
\begin{equation}\label{ldw}
\mathcal{L}_X g= 2 \lambda' g\,, 
\end{equation}
which is the same as saying that $X$ is conformal with potential function $\lambda'$; and the second is that for any vector  field $Y \in T_pW$, we have 
\begin{equation}\label{dcw}
\nabla_YX = \lambda' Y\,,
\end{equation}
where $\nabla$ denotes the Riemannian connection of $(M,g).$ They proved the following result:

\begin{theorem}[\cite{HaXa}]\label{teoHaXa}
  Let $\Sigma$ be a compact free boundary minimal surface in the geodesic ball $B_R=[0,R)\times \mathbb{S}^2\subset W$ with radius $R<R_{\infty}.$ Assume that for all points $x\in\Sigma,$
  \begin{eqnarray}\label{eq002}
\dfrac{|A|^2\langle N(x), X\rangle^2}{(\lambda')^2}\leq 2,
  \end{eqnarray}
where $N(x)$ denotes a unit normal vector at the point $x\in \Sigma,$ $X=\lambda(r)\partial_r$ is the conformal vector field on $W$ and $A$ denotes the second fundamental form on $\Sigma,$ then
\begin{enumerate}[i)]
 \item either $|A|^2\langle N(x), X\rangle^2\equiv 0$ and $\Sigma$ is a totally geodesic disk,
 \item or $\dfrac{|A|^2\langle N(x), X\rangle}{(\lambda')^2}= 2,$ at some point $p\in\Sigma$ and $\Sigma$ is a rotational annulus.
\end{enumerate}
 \end{theorem}
 To prove the three results above, in \cite{ambrozio2016gap}, \cite{barbosa2019gap} and \cite{HaXa}, were obtained the appropriate gap conditions to assure that a certain function defined in a surface was convex. Inspired in their ideas, we consider $\Sigma $ a surface isometrically immersed in $(\B_r^{3},\bar{g}),$ where $\bar{g}=e^{2h}\left\langle , \right\rangle,$ $h(x)=u(\left|x\right|^2)$ for some smooth function $u$ defined on $I=[0,r^2).$ In this case, the vector field $\vec{x}$ is conformal in $(\B_r^{3},\bar{g}),$ and the equations (\ref{ldw}) and (\ref{dcw}) becomes, 
\begin{equation}\label{ldc}
\mathcal{L}_{\vec{x}} \bar{g}= 2 \sigma \bar{g} 
\end{equation}
and 
\begin{equation}\label{dcc}
\nabla_Y\vec{x} = \sigma Y\,,
\end{equation}
where $\nabla$ denotes the Riemannian connection of $(\B^3_r,\bar{g})$ and $\sigma= 1 + 2u'(\vert x \vert^2)\vert x \vert^2$. We choose $r>0$ so that the function $\sigma:\B^3_r \rightarrow \R$  
is positive.

Motivated by those works, we can state our main result as follows:

\begin{theorem}\label{classifcmc}
 Let $\Sigma$ be a compact free boundary CMC surface in $(\B_r^{3},\bar{g})$. Suppose that for all points $x\in\Sigma$,
 \begin{equation}\label{ccmc1}
\left\{
\begin{array}{rcll}
 \dfrac{|\Phi|^2}{\sigma^2} \bar{g}(\vec{x},N)^2 &\leq& \dfrac{1}{2}\left(2+\dfrac{H}{\sigma}\bar{g}(\vec{x},N)\right)^2\\
 0&\leq&2+\dfrac{H}{\sigma}\bar{g}(\vec{x},N).
\end{array}
\right.
\end{equation}
Then, one of the following situations occurs,
\begin{enumerate}[i)]
 \item either $\Sigma$ is diffeomorphic to a disk, 
 \item or $\Sigma$ is rotationally symmetric with nontrivial topology.
\end{enumerate}
\end{theorem}

Note that if $\Sigma$ is a minimal surface, then $|\Phi|^2=|A|^2$. Moreover, we do not need the second hypothesis in the Theorem \ref{classifcmc}. Thus, we have the following result: 

\begin{corollary}\label{teoprin2h0}Let $\Sigma$ be a compact free boundary minimal surface in $(\B_r^{3},\bar{g})$. Suppose that for all points $x\in\Sigma$,
\begin{eqnarray}
\label{gapmin}
\frac{|A|^2}{\sigma^2} \bar{g}(\vec{x},N)^2 &\leq& 2.
\end{eqnarray}
Then, one of the following situations occurs,

\begin{enumerate}[i)]
\item either $\Sigma$ is diffeomorphic to a disk, 

\item or $\Sigma$ is rotationally symmetric with nontrivial topology.
\end{enumerate}
\end{corollary}

Another consequence of our main Theorem is the following Corollary.

\begin{corollary}\label{corodesiestr} Let $\Sigma$ be a compact free boundary CMC surface in $(\B_r^3,\bar{g})$. If 
$$
\dfrac{|\Phi|^2}{\sigma^2} \bar{g}(\vec{x},N)^2 < \dfrac{1}{2}\left(2+\dfrac{H}{\sigma}\bar{g}(\vec{x},N)\right)^2,
$$
then $\Sigma$ is diffeomorphic to a disk $\mathbb{D}^2$.  
\end{corollary}

Our approach is to consider $\Sigma \subset (\B_r^3,\bar{g})$ instead of  $\Sigma \subset (M,dr^2+\lambda(r)^2g_{\S^2})$ as in the Theorem \ref{teoHaXa} and includes the spaces considered by Li and Xiong in \cite{HaXa} and has a subtle advantage. In \cite{HaXa}, the authors considered the manifold $W$ as a model of the hyperbolic space $\H^3$ or  the semisphere $\S^3_+$ being submanifold of the Minkowski space $(\mathbb{R}^4_1,\left\langle\cdot,\cdot\right\rangle_1)$, where  a certain Jacobi function, induced by rotation around two axis, were already known. Under the gap condition (\ref{eq002}) in Theorem \ref{teoHaXa}, the surface is invariant by theses rotations. However, it is not always possible to see a warped product $W$ as a submanifold of the Minkowski space for any warping function $\lambda$. When we consider a CMC surface $\Sigma$ in $(\B_r^3,\bar{g})$ as previous, the Jacobi function induced by rotation can be constructed in natural way just exploring the fact that conformal fields remains conformal under conformal change on Euclidean metric as will become clear later.



The rest of this paper is organized as follows. In the Section \ref{sec:preliminares}, we describe the main tools to prove our results. Then, in Section \ref{sec:theoremain}, we prove the Theorem \ref{classifcmc}. And finally, we give an example to illustrate the Theorem \ref{classifcmc} in Section \ref{sec:examples}.

\section{Preliminares}\label{sec:preliminares}


%
%
%
Let $u:\left[0,a^2\right) \rightarrow \R $ be a smooth function where $a\leq \infty$. Consider the function $h:\B_{{a}}^{3} \rightarrow \R$ given by $h(x)=u(\left|x\right|^2)$ defined on the Euclidean ball $\B_{{a}}^{3}$ of radius $a$  and centred at the origin, where we assume that $\B_{{a}}^{3}:=\R^3$ when $a=\infty$. Let $\bar{g}$ be a metric obtained from a conformal change in Euclidean metric $\left\langle ,\right\rangle$ given by
\begin{equation}\label{metricaconforme}
\bar{g}=e^{2h}\left\langle ,\right\rangle. 
\end{equation}

Fix $r<a$ and let $(\B_{r}^{3},\bar{g})$ be a submanifold of $(\B_{a}^{3},\bar{g})$ with canonical coordinates $x=(x_1,x_2,x_{3})$ and conformal metric $\bar{g}$ given by $(\ref{metricaconforme})$.
The position vector field is a vector field such that for each point $x\in (\B_r^{3},\bar{g})$ associates the vector $\sum x_i \partial_i$, which is denoted by  $\vec{x}$. 

The vector field $\vec{x}$ in $(\B_r^{3},\left\langle\,,\,\right\rangle)$ is conformal, i.e, the Lie derivative of $\left\langle\,,\,\right\rangle$ with respect to $\vec{x}$ satisfies 
$$
\mathcal{L}_{\vec{x}} \left\langle\,,\,\right\rangle = 2f\left\langle  , \right\rangle, 
$$
where $f$ is a scalar function, called the potential function, and satisfies $f\equiv 1$. Under conformal change $\bar{g}=e^{2h}\left\langle\,,\,\right\rangle$, the vector field $\vec{x}$ remains conformal and in this case the Lie derivative of $\bar{g}$ with respect to $\vec{x}$ is given by 
\begin{equation}\label{xconf}
\mathcal{L}_{\vec{x}} \bar{g} = 2\sigma \bar{g},
\end{equation}
where $\sigma(x) = f + \vec{x}(h) = 1 + 2u'(\left| x \right|^2)\left|\vec{x}\right|^2$. From now on, $\sigma$ always will denote the potential function of the conformal vector field $\vec{x}$ with respect to $\bar{g}$.

Concerning the Euclidean metric, the gradients of the functions $ h $ and $ e^{2 h}$ are given, respectively, by 
\begin{equation}\label{gradh}
\mbox{\texttt{grad}}(h)  = 2u'(\left|x\right|^2)\vec{x} \quad \mbox{and} \quad \mbox{\texttt{grad}}(e^{2h}) = 4e^{2h}u'(\left|x\right|^2)\vec{x}.
\end{equation}

\begin{example} Consider the function $u:[0,1)\rightarrow \R$ given by $u(t)=\ln \left(\dfrac{2}{1-t}\right)$. Then, $(\B_1^{3},\bar{g})$ is the hyperbolic space $\mathbb{H}^{3}$ modelled on the Poincar\'e  disk.  
\end{example}
\begin{example} Let $u:[0,\infty)\rightarrow \R$ be the function given by $u(t)=\ln \left(\dfrac{2}{1 + t}\right)$. Then, $(\R^{3},\bar{g})$ is the space $\S^{3}\setminus \{p\}$, i.e,  the sphere minus one pole where the origin $\vec{0} \in \R^3$ can be interpreted as the other pole. 
\end{example}
\begin{example} Let $u:[0,\infty) \rightarrow \R$ be a function given by $u(t)= -\dfrac{t}{2n}$. 
 Then, $(\R^{n+1},\bar{g})=\left(\R^{n+1}, e^{-\frac{\left|\vec{x}\right|^2}{2n}}\left\langle ,\right\rangle\right)$ is the Riemannian manifold often referred as the Gaussian Space. 
\end{example}

\begin{example} If $u:[0,\infty)\rightarrow \R$ is given by $u(t)=0$, then $(\R^{n+1},\bar{g})$ is the Euclidean space with the canonical metric. 
\end{example}
%
%

Note that if we fix a point $x_{_0} \in (\B_r^{3},\bar{g})$, then the straight segment connecting the origin $\vec{0}$ to $x_{_0}$ is a geodesic with respect to the metric $\bar{g}$. If $r$ is the canonical distance of a point $x\in\B_a^{3}$ to the origin, then the distance $\bar{r}$ of $x$ to $\vec{0}$ in $(\B_a^{3},\bar{g})$ is the length of the straight segment $\gamma$ which connects $\vec{0}$ to $x$ with respect to the metric $\bar{g}$. If $\gamma(t)=tx$ is a parametrization of $\gamma$, then $\gamma'(t)=\vec{x}$. Thus, 
\begin{equation}\label{eqdistconf}
\bar{r}=\int_0^1\sqrt{\bar{g}(\vec{x},\vec{x})}dt=\int_0^1 \left|x\right|e^{u(\left|tx\right|^2)}dt= r \int_0^1 e^{u(t^2r^2)}dt.
\end{equation}

If $(\B_1^3,\bar{g})$ is the hyperbolic space $\mathbb{H}^3$, then the hyperbolic distance $\bar{r}_{_\mathbb{H}}$ of a point $x\in\mathbb{H}^3$ to the origin is given by 
$$
\bar{r}_{_\mathbb{H}}=r \int_0^1 \frac{2}{1-t^2r^2}dt = 2\tanh^{-1}(r),
$$
where $r=\left|x\right|$ is the euclidean distance. Similarly, if $(\R^3,\bar{g})$ is the $\S^3\setminus \{p\}$, then the spherical distance $\bar{r}_{_\mathbb{S}}$ of a point $x\in\S^3\setminus \{p\}$ at antipodal point of $p$ is given by
$$
\bar{r}_{_\mathbb{S}}=r \int_0^1 \frac{2}{1+t^2r^2}dt = 2\tan^{-1}(r)\,.
$$

Let $\overline{\nabla}$ and $\nabla$ be the Levi-Civita connections of 
$(\mathbb{B}_r^{3},\bar{g})$ and $(\mathbb{B}_r^{3},\left\langle , \right\rangle)$, respectively. Using the Kozul formula, we can obtain that
\begin{equation}\label{conconf}
\overline{\nabla}_YX=\nabla_YX+ Y(h)X+X(h)Y - \left\langle X,Y\right\rangle\nabla h\,
\end{equation}
for any vector fields $X,Y \in \mathcal{X}(\B_r^3)$. In this setting, we have the following result:

\begin{lemma}\label{conexaoconf} Let $(\B_r^{3},\bar{g})$ be the Euclidean ball with a metric $\bar{g}=e^{2h}\left\langle , \right\rangle$ defined as before. 
Then, we have  

\begin{enumerate}[i)]

\item $\overline{\nabla}_Y \vec{x} = \sigma Y $, where $\sigma=1+2u'(\left|\vec{x}\right|^2)\left|\vec{x}\right|^2$

\item $\overline{\nabla} \sigma= 4e^{-2h}(u''(\left|x\right|^2)\left|x\right|^2 + u'(\left|x\right|^2))\vec{x}$,
\end{enumerate}

\noindent where $\overline{\nabla} \sigma$ denotes the gradient of the potential function $\sigma$.
\end{lemma}

\begin{proof} $i)$ Using the expression of $\overline{\nabla}$ given by $(\ref{conconf})$ we have
\begin{eqnarray*} 
\overline{\nabla}_Y\vec{x}&=& \nabla_Y\vec{x} + Y(h)\vec{x} + \vec{x}(h)Y - \left\langle \vec{x},Y \right\rangle \nabla h \\
                     &=& Y+\left\langle \nabla h,\vec{x}\right\rangle Y+\left\langle \nabla h,Y\right\rangle \vec{x} - \left\langle \vec{x},Y\right\rangle \nabla h \\ 
										 &=& Y+2u'(\left|x\right|^2)\left\langle \vec{x},\vec{x}\right\rangle Y+2u'(\left|x\right|^2)\left\langle \vec{x},Y\right\rangle \vec{x} - 2u'(\left|\vec{x}\right|^2)\left\langle \vec{x},Y\right\rangle \vec{x} \\
										 &=& (1+2u'(\left|\vec{x}\right|^2)\left|\vec{x}\right|^2)Y\\
										 &=& \sigma Y.
\end{eqnarray*}
\noindent $ii)$ If \texttt{grad}$(\sigma)$ denote the gradient of the function $\sigma$ with respect to the metric $\left\langle ,\right\rangle$, then $\overline{\nabla} \sigma=e^{-2h}\mbox{\texttt{grad}}(\sigma)$. Calculating  \texttt{grad}$(\sigma)$, we obtain
\begin{eqnarray*}
\mbox{\texttt{grad}}(\sigma) &=& \sum_{i=1}^{3}{\frac{\partial}{\partial{x_i}}(1+2u'(\left|x\right|^2)\left|x\right|^2)\frac{\partial}{\partial{x_i}}}\\
							&=& \sum_{i=1}^{3}{4(u''(\left|x\right|^2)\left|x\right|^2 + u'(\left|x\right|^2))x_i\frac{\partial}{\partial{x_i}}}\\
							&=& 4(u''(\left|x\right|^2)\left|x\right|^2 + u'(\left|x\right|^2))\vec{x}.
\end{eqnarray*}
Thus, 
\begin{equation*}
\overline{\nabla} \sigma= 4e^{-2h}(u''(\left|x\right|^2)\left|x\right|^2 + u'(\left|x\right|^2))\vec{x}. 
\end{equation*}
\end{proof}

The next proposition was inspired in \cite{HaXa} and has a subtle importance. In short, it allows simplify calculations in order to construct a function that becomes convex under a gap geometric hypothesis as will be clear in {Lemma \ref{authess}}. 

\begin{proposition}\label{lemasolfund} Let $a,b:[0,r_0) \rightarrow \R $ $(r_0 \leq \infty)$ be smooth functions such that:
\begin{enumerate}[i)]

\item $b(t)>0$ $\forall\,t\,\in [0,r_0);$

\item $a(0)=0$, $a(t)>0$, $\forall\,t\in(0,r_0);$

\item $a'(t)>0$, $\forall\,t\in(0,r_0)$. 
\end{enumerate}

\noindent Then, the differential equation 
\begin{equation}\label{eqdiffegeral}     
\Phi''(a)a'b + \Phi'(a)b'=0
\end{equation}
admits a solution $\Phi:(0,\tilde{r}_0) \longrightarrow \R $ that satisfies  $\Phi'(s)>0$, where $(0,\tilde{r}_0)$ is a image of the interval $(0,r_0)$ by the function $a$.
\end{proposition}
\begin{proof} We define $f(t):=\Phi'(a(t))$. Then $f'(t)=\Phi''(a(t))a'(t)$. Thus, the equation $(\ref{eqdiffegeral})$ is equivalent to 
$$
f'(t)b(t) + f(t)b'(t)=0.
$$
Solving the above equation on $t$, we obtain
$$
f(t)=\frac{c_1}{b(t)}, 
$$
namely, 
$$
\Phi'\big(a(t)\big)=\frac{c_1}{b(t)},
$$
where $c_1>0.$
Multiplying both sides of the above equation by $a'(t)$, and integrating with respect to $t$, we obtain
$$
\int_0^r {\Phi'\big(a(t)\big)a'(t)}dt=c_1 \int_0^r \frac{a'(t)}{b(t)}dt. 
$$
As $a'(t)>0$ for $0 < t < r_0 $, the function $a:[0,r_0) \rightarrow [0,\tilde{r}_0)$ is a diffeomorphism in $(0,r_0)$. Thus, we can consider the change of variable $a(t)=\xi$. Therefore, $d\xi=a'(t)dt$, which results in
$$ 
\int_0^{s} {\Phi'(\xi)}d\xi=c_1 \int_0^{s} \frac{1}{b(a^{-1}(\xi))}d\xi,
$$
where $s=a(r) \in [0,\tilde{r}_0)$. By the Fundamental Theorem of Calculus,
$$
\Phi(s)-\Phi(0)=c_1 \int_0^{s} \frac{1}{b(a^{-1}(\xi))}d\xi.
$$
Choosing $\Phi(0)=0$ and $c_1=1$, the function
$$
\Phi(s)= \int_0^{s} \frac{1}{b(a^{-1}(\xi))}d\xi
$$
is a solution to the equation $(\ref{eqdiffegeral})$ which satisfies $\Phi'(s)>0,\,\,\forall\,s \in [0,\tilde{b})$,  as desired. 
 
\end{proof}
In the next result we calculated the Hessian of the function $\varphi:\Sigma \rightarrow \R$ given by $\varphi(x)=\bar{g}(\vec{x},\vec{x}).$

\begin{lemma}\label{henssiano} Let $\Sigma$ be a smooth surface in $(\B_r^{3},\bar{g})$. Define the function $\varphi:\Sigma \rightarrow \R$ by $\varphi(x)=\bar{g}(\vec{x},\vec{x})$. Then, 
$$
Hess_{\Sigma}\varphi(x)(Y,Z)=2 \left( \bar{g}(\overline{\nabla}\sigma,Y)\bar{g}(\vec{x},Z)  + \sigma^2\bar{g}(Z,Y)+ \sigma\bar{g}(A(Y),Z)\bar{g}(N,\vec{x})\right),
$$ 
\noindent where $\overline{\nabla} \sigma$ denote the gradient of the potential function of the conformal vector field $\vec{x}$ with respect to $\bar{g},$ $N$ denotes the unit normal vector field at the point $x\in\Sigma$ and $A$ is the operator form on $\Sigma.$
\end{lemma}

\begin{proof} Let $\nabla$ be the connection of $\Sigma$. Using the Lemma \ref{conexaoconf}, we have that $Hess_{\Sigma}\varphi$ is given by 
\begin{eqnarray*}
Hess_{\Sigma}\varphi(x)(Y,Z) &=& YZ(\varphi) - (\nabla_YZ)(\varphi)\\
														 &=& YZ\bar{g}(\vec{x},\vec{x}) - (\nabla_YZ)\bar{g}(\vec{x},\vec{x})	\\
														 &=& 2 \left(Y\bar{g}(\overline{\nabla}_Z\vec{x},\vec{x}) - \bar{g}(\overline{\nabla}_{_{\nabla_{_{Y}}Z}}\vec{x},\vec{x})\right)\\
														 &=& 2 \left(Y\bar{g}(\sigma Z,\vec{x}) - \bar{g}(\sigma \nabla_YZ,\vec{x})\right)\\
														 &=& 2 \left( Y(\sigma)\bar{g}(Z,\vec{x})+ \sigma \bar{g}(\overline{\nabla}_YZ,\vec{x}) + \sigma^2\bar{g}(Z,Y)- \sigma \bar{g}(\nabla_YZ,\vec{x})  \right) \\
														&=& 2 \left( \bar{g}(\overline{\nabla} \sigma, Y)\bar{g}(Z,\vec{x})+ \sigma^2\bar{g}(Z,Y)+ \sigma \bar{g}(\overline{\nabla}_YZ-\nabla_YZ ,\vec{x})\right) \\
														&=& 2 \left( \bar{g}(\overline{\nabla} \sigma, Y)\bar{g}(Z,\vec{x})+ \sigma^2\bar{g}(Z,Y)+ \sigma\bar{g}(A(Y),Z)\bar{g}(N,\vec{x})\right)
\end{eqnarray*}
as desired.
\end{proof}

\begin{lemma}\label{authess} Let $\Sigma$ be a surface in $(\B^3_r,\bar{g}).$ Let $\Psi:\Sigma\rightarrow\R$ be a function defined by  
\begin{equation}\label{funcconvex}
\Psi(x)=\Phi(\varphi(x)),
\end{equation}
where $\varphi(x)=\bar{g}(\vec{x},\vec{x})=e^{2u(\left|x\right|^2)}\left|\vec{x}\right|^2$ and $\Phi$ is a solution of the equation $(\ref{eqdiffegeral})$ given by Proposition \ref{lemasolfund}. Then, the eigenvalues of $Hess_{\Sigma}\Psi(x)$ are given by
$$\overline{\lambda}_1=2\sigma^2\Phi'\left(1+\dfrac{\overline{k}_1}{\sigma}\overline{g}(\vec{x},N)\right)\ \text{and}\  \overline{\lambda}_2=2\sigma^2\Phi'\left(1+\dfrac{\overline{k}_2}{\sigma}\bar{g}(\vec{x},N)\right),$$
here $\overline{k}_1$ and $\overline{k_2}$ are the principal curvatures of $\Sigma$ with respect to the normal vector $N.$
\begin{proof} It follows from the Lemma \ref{conexaoconf} that $\overline{\nabla}_Y\vec{x}=\sigma Y$, where $\sigma = 1+2u'(\left|x\right|^2)\left|\vec{x}\right|^2$. We choose 
$$a(t)=e^{2u(t)}t\,\, \mbox{,}\,\, b(t)=1+2u'(t)t$$
and consider a solution $\Phi$ of the equation $(\ref{eqdiffegeral})$ given by Proposition \ref{lemasolfund}:
\begin{equation}\label{eqfundam001}
\Phi''(e^{2u(t)}t)e^{2u(t)}(1+2u'(t)t)^2 + 2\Phi'(e^{2u(t)}t)(u''(t)t+u'(t))=0.
\end{equation}
Define $\Psi:\Sigma\rightarrow\R$ by  
\begin{equation}\label{funcconvex}
\Psi(x)=\Phi(\varphi(x)),
\end{equation}
where $\varphi(x)=\bar{g}(\vec{x},\vec{x})=e^{2u(\left|x\right|^2)}\left|\vec{x}\right|^2$. Thus,
\begin{eqnarray}\label{henscompo}
Hess_{\Sigma}\Psi(x)(Y,Z) &=& YZ(\Phi(\varphi)) - \nabla_YZ(\Phi(\varphi)) \nonumber\\
                          &=& Y(\Phi'(\varphi)Z(\varphi)) - \Phi'(\varphi)\nabla_YZ(\varphi) \nonumber \\
                          &=& Y(\Phi'(\varphi))Z(\varphi) + \Phi'(\varphi)YZ(\varphi) - \Phi'(\varphi) \nabla_YZ(\varphi) \nonumber \\
													&=& \Phi''(\varphi)Y(\varphi)Z(\varphi) + \Phi'(\varphi)Hess_{\Sigma} \varphi(x)(Y,Z) \nonumber \\
													&=& 4\sigma^2\Phi''(\varphi)\bar{g}(Y,\vec{x})\bar{g}(Z,\vec{x}) + \Phi'(\varphi)Hess_{\Sigma} \varphi(x)(Y,Z).\nonumber 
\end{eqnarray}
Using the expression of $Hess_{\Sigma} \varphi(x)$ obtained in the previous lemma, we have
\begin{eqnarray}\label{hesscomposto}
Hess_{\Sigma}\Psi(x)(Y,Z) &=& 4\sigma^2\Phi''(\varphi)\bar{g}(Y,\vec{x})\bar{g}(Z,\vec{x})\nonumber \\
                          & & + 2\Phi'(\varphi)\left(\bar{g}(\bar{\nabla} \sigma, Y)\bar{g}(Z,\vec{x})+ \sigma^2\bar{g}(Z,Y)+ \sigma\bar{g}(A(Y),Z)\bar{g}(N,\vec{x})\right)\nonumber \\
                          &=& \left\{4\sigma^2\Phi''(\varphi) + 2\Phi'(\varphi)G(x)\right\}\bar{g}(Y,\vec{x})\bar{g}(Z,\vec{x}) \nonumber \\
													& & + 2\sigma^2\Phi'(\varphi)\left(\bar{g}(Z,Y)+ \frac{1}{\sigma} \bar{g}(A(Y),Z)\bar{g}(N,\vec{x})\right), 
\end{eqnarray} 
where $G(x)=4e^{-2u(\left|x\right|^2)}(u''(\left|\vec{x}\right|^2)\left|\vec{x}\right|^2 + u'(\left|\vec{x}\right|^2)$. Therefore, by $(\ref{eqfundam001})$, we obtain 
$$
\left\{4\sigma^2\Phi''(\varphi) + 2\Phi'(\varphi)G(x)\right\} \equiv 0.
$$
Thus, (\ref{hesscomposto}) provides:
\begin{eqnarray*}
Hess_{\Sigma}\Psi(x)(Y,Z) = 2\sigma^2\Phi'(\varphi)\mathcal{B}_0(Y,Z),
\end{eqnarray*}
where $\mathcal{B}_0(Y,Z):=\bar{g}(Z,Y)+ \frac{1}{\sigma} \bar{g}(A(Y),Z)\bar{g}(N,\vec{x})$. 

Let $\{E_1,E_2\}$ be a local frame that diagonalises the second fundamental form, i.e, $A(E_i)=\bar{k}_iE_i$, where $\bar{k}_i$ is a principal curvature. Notice that the eigenvectors of $A$ are also eigenvectors of the symmetric bilinear form $\mathcal{B}_0(\cdot,\cdot)$, thus the eigenvalues of $\mathcal{B}_0(\cdot,\cdot)$ are 
$$
\lambda_i=1 + \dfrac{\bar{k}_i}{\sigma}\bar{g}(\vec{x},N).
$$
\end{proof}
\end{lemma}

\begin{proposition} \label{hesnong}
 Under the same hypothesis in Theorem \ref{classifcmc}, we have
 $$
Hess_{\Sigma}\Psi(p)(Y,Y)\geq 0,\,\, \forall \,\, Y \in T_p\Sigma \,\, and \,\, p \in \Sigma.
$$
\end{proposition}

\begin{proof}
 To prove that $Hess_{\Sigma}\Psi(p)(Y,Y)\geq 0$ is equivalent to show that $\overline{\lambda}_1$ and $\overline{\lambda}_2$ are nonnegative, and since $\Phi'(\varphi)>0$ and $\sigma>0,$ then $2\sigma^2\Phi'(\varphi)>0.$ So, this is equivalent to prove that $\lambda_1$ and $\lambda_2$ are nonnegative too. In fact,
 \begin{eqnarray}
  {\lambda_1}{\lambda_2}&=&\left(1+\dfrac{\overline{k}_1}{\sigma}\overline{g}(\vec{x},N)\right)\left(1+\dfrac{\overline{k}_2}{\sigma}\overline{g}(\vec{x},N)\right)\\
  &=&1+\dfrac{\overline{H}}{\sigma}\overline{g}(\vec{x},N)+\dfrac{\overline{k}_1\overline{k}_2}{\sigma^2}\overline{g}(\vec{x},N)^2\nonumber\\
  &=&1+\dfrac{\overline{H}}{\sigma}\overline{g}(\vec{x},N)+\dfrac{\overline{H}^2-|A|^2}{2\sigma^2}\overline{g}(\vec{x},N)^2\nonumber\\
  &=&1+\dfrac{\overline{H}}{\sigma}\overline{g}(\vec{x},N)-\dfrac{|\Phi|^2}{2\sigma^2}\overline{g}(\vec{x},N)^2+\dfrac{\overline{H}^2}{(2\sigma)^2}\overline{g}(\vec{x},N)^2\nonumber\\
  &=&\dfrac{1}{4}\left(2+\dfrac{\overline{H}}{\sigma}\overline{g}(\vec{x},N)\right)^2-\dfrac{|\Phi|^2}{2\sigma^2}\overline{g}(\vec{x},N)^2.\nonumber
 \end{eqnarray}
Now, using the first inequality in (\ref{ccmc1}), we obtain that $\lambda_1\lambda_2\geq0$, is that, $\lambda_1$ and $\lambda_2$ have the same sign. Using the second inequality in Theorem \ref{classifcmc} we have $$\lambda_1+\lambda_2=2+\dfrac{\overline{H}}{\sigma}\overline{g}(\vec{x},N)\geq 0.$$ Therefore, $\lambda_{i}\geq 0,\ i\in\{1,2\}.$
\end{proof}

\begin{lemma}\label{convex}
 Under the same hypothesis in Theorem \ref{classifcmc}, we have
 \begin{enumerate}[i)]
  \item The geodesic curvature of $\partial \Sigma$ is $\overline{k}_{ge}=\dfrac{\sigma}{e^{u(r^2)}}.$ In particular, as  $\sigma>0$ in $\B^3_r,$ $\partial\Sigma$ is strictly convex.
  \item The set 
  \begin{equation}\label{conjC}
\mathcal{C}=\left\{p\in\Sigma; \Psi(p)=min_\Sigma\Psi(x)\right\},
\end{equation}
is totally convex, i.e. any geodesic arc $\gamma$ joining two points in $\mathcal{C}$ is entirely contained in $\mathcal{C}$. 
 \end{enumerate}
 
\begin{proof}
 i) The geodesic curvature $k_{g_e}$ of $\partial \Sigma$ with respect to the immersion $\Sigma$ in the ball $(\B_r^3,\left\langle\,,\,\right\rangle)$ is given by $k_{g_e}=\frac{1}{r}$. Then, by Lemma 10.1.1 in \cite{lopez2013constant}, we obtain that the geodesic curvature in the metric $\bar{g}$ is 
\begin{equation}\label{curvgeomudconf0}
\bar{k}_{g_e}=e^{-h}\left(k_{g_e} - \nu(h) \right)\,,
\end{equation}
where $\nu$ denotes the inner unit normal vector of $\partial \Sigma$ with respect to the metric $\left\langle\,,\,\right\rangle$. We prove here only by completeness. In fact, let $\bar{v}$ and $\bar{\nu}$ be  a unit vector field over $\partial \Sigma$ with respect to metric $\bar{g}$ in such way that $\bar{v}$ is tangent to $\partial \Sigma$ and $\bar{\nu}$ is inner normal to $\partial \Sigma$. We have,
\begin{eqnarray}\label{curvgeomudconf}
\bar{k}_{g_e}&=&\bar{g}(\bar{\nabla}_{\bar{v}}\bar{v},\bar{\nu} ) = \bar{g}( \nabla_{\bar{v}}\bar{v} + 2\bar{v}(h)\bar{v} - \left\langle \bar{v},\bar{v} \right\rangle \texttt{grad}(h),e^{-h}\nu ) \nonumber \\
             &=& \bar{g}( \nabla_{\bar{v}}\bar{v} - \left\langle \bar{v},\bar{v} \right\rangle \texttt{grad}(h),e^{-h}\nu )\nonumber \\
						 &=& \bar{g}( \nabla_{e^{-h}v} e^{-h}v - \left\langle \bar{v},\bar{v} \right\rangle \texttt{grad}(h),e^{-h}\nu )\nonumber \\ 
						 &=& e^{-h}\left(e^{-2h}\bar{g}(\nabla_{v} v,\nu) - (e^{-2h} \bar{g}( \texttt{grad}(h),\nu )\right)\nonumber \\
						 &=& e^{-h}\left(\left\langle \nabla_{v} v,\nu\right\rangle -  \left\langle \texttt{grad}(h),\nu \right\rangle \right)\nonumber \\
						 &=& e^{-h}\left(k_{g_e} -  \left\langle \texttt{grad}(h),\nu \right\rangle \right)\nonumber \\
						 &=& e^{-h} \left( k_{g_e} - \nu(h) \right).
\end{eqnarray}
On the other hand,
$$
\nu(h)= \left\langle \texttt{grad}(h),\nu \right\rangle = 2u'(r^2)\left\langle \vec{x},\nu \right\rangle = 2u'(r^2)\left\langle -r\nu,\nu \right\rangle = - 2u'(r^2)r. 
$$
Therefore,
$$
\bar{k}_{g_e}=e^{-u(r^2)}\left(\frac{1}{r} + 2u'(r^2)r \right)=\frac{e^{-u(r^2)}}{r}\left(1 + 2u'(r^2)r^2 \right) = \frac{\sigma}{e^{u(r^2)}r}.
$$ 
As $\sigma>0$, we have $\bar{k}_{g_e}>0$.  

To prove ii) we consider $p_1,\ p_2 \in \mathcal{C}.$ Since $\partial\Sigma$ is strictly convex, there is a geodesic path $\gamma:[0,1]\to\Sigma$ with $\gamma(0)=p_1$ and $\gamma(1)=p_2.$ Now, we consider $h(t)=\Phi(\varphi(\gamma(t)))$, then by Proposition \ref{hesnong} we have $Hess_{\Sigma}\Phi\geq 0,$  and therefore $h''\geq 0.$ Since $h$ attains its minimum at $t=0$ and $t=1$, we get $h(t)\equiv\displaystyle\min_{x\in\Sigma}\Phi(\varphi(x))$ and therefore $\gamma([0,1])\subset \mathcal{C}.$
\end{proof}
\end{lemma}

\begin{corollary}\label{pointsc}
 Under the same hypothesis in Theorem \ref{classifcmc}, we have
 \begin{enumerate}[i)]
  \item either $\mathcal{C}$ contains a single point $p\in\mathcal{C}$,
  \item or $\mathcal{C}$ contains more than one point and in this case, any points $p_1$ and $p_2$ $\in \Sigma$ can be connected by a minimizing geodesic $\gamma:\left[0,1\right]\rightarrow \Sigma $ such that $\gamma\left([0,1\right]) \subset \mathcal{C}$. Therefore, $\mathcal{C}$ is a connected set. 
\end{enumerate}

\end{corollary}

 \begin{proof}
 Since $\Psi:\Sigma \rightarrow \R$ is a continuous function and $\Sigma$ is compact, then the set $\mathcal{C}$ given in \eqref{conjC} is not empty. If $\mathcal{C}$ contains a single point we have $i)$. 

Now, suppose that $\mathcal{C}$ contains more than one point, say $p_1$ and $p_2$ with $p_1\neq p_2$. Then, there is a minimizing geodesic $\gamma:[0,1]\to \Sigma$ such that $\gamma(0)=p_1$ and $\gamma(1)=p_2,$ because $\partial \Sigma$ is strictly convex.  By the item $ii)$ of Lemma \ref{convex} we have $\gamma([0,1])\subset \mathcal{C}$ and thus $\mathcal{C}$ is a connected set.
\end{proof}

\section{Proof of the Theorem \ref{classifcmc}}\label{sec:theoremain}

 \begin{proof} (Theorem \ref{classifcmc}) Let $\mathcal{C}$ be a set considered in Lemma \ref{convex}. First, suppose that $\mathcal{C}$ contains a single point. We claim that $\Sigma$ is a topological disk. In fact, we prove this affimation using the same idea as in \cite{ambrozio2016gap}.

Let $[\alpha]$ be a given homotopy class in $\Sigma$ with basis at $p\in\mathcal{C}.$ Suppose that $[\alpha]$ is a nontrivial homotopy class. Since $\partial \Sigma$ is strictly convex, we can find a geodesic loop $\gamma:[0,1]\to\Sigma$ with $\gamma(0)=\gamma(1)=p$ and $\gamma\in[\alpha].$ Observe that, $\gamma([0,1])\subset\mathcal{C},$ because $\mathcal{C}$ is totally convex. Since $\mathcal{C}=\{p\}$ and $[\alpha]$ is nontrivial class we have a contradiction. Therefore $\pi_1(\Sigma,p)=0$ and we conclude that $\Sigma$ is a topological disk.


Now, suppose that $\mathcal{C}$ contains more than a single point. Let $\gamma:[0,1]\rightarrow\mathcal{C}$ be a minimizing geodesic  joining two points in $\mathcal{C}$. By the item $ii)$ of Corollary \ref{pointsc} we have $\gamma([0,1])\subset \mathcal{C}$. Consider $c=\mbox{min}_\Sigma\Psi(x)$ and denote by $\overline{\nabla} \Psi $ the gradient of the function $\Psi$ in $(\B_r^3,\bar{g})$, and by $\nabla \Psi$ the gradient of $\Psi$  when restricted to $\Sigma$. Since $\gamma(t)$ is a critical point of $\Psi_{\,|_{\Sigma}}$ for each $t\in[0,1]$, we have that $\nabla \Psi = 0,$ along $\gamma,$ that is, there is a function $n(t)$ such that $\overline{\nabla}\Psi=n(t)N$, because $\nabla \Psi = (\overline{\nabla} \Psi)^T$. Since the function $\Psi$ is radial, the vector field $\overline{\nabla}\Psi$ is normal to the level sets of $\Psi$ that in this case, are spheres. Consider the sphere $\S_{\lambda}^2$ with radius $\lambda<r$ centred at the origin which satisfies $\Psi(x) = c$, $\forall \,x$ $\in$ $\S_{\lambda}^2$. As $\overline{\nabla}\Psi=n(t)N$ in $\gamma(t)$, we have that $N$ is normal to $\S_{\lambda}^2$ at the points $x=\gamma(t)$. As $\gamma(t)$ is a geodesic in $\Sigma$, it is also a geodesic in $\S_{\lambda}^2$. Therefore, $\gamma(t)$ is an arc of the great circle.

Let $\pi \subset \R^3$ be the plane passing through the origin such that  $\gamma \subset \pi$, and let $E$ be a unit normal vector in $\R^3$ orthogonal to $\pi$. Consider the vector field 
$$V = \vec{x} \wedge E .$$
The Lie derivative of $<,>$ in direction to $V$ satisfies $\mathcal{L}_{V}\left\langle, \right\rangle=2f\left\langle, \right\rangle$, where $f\equiv0$. Thus, $V$ is a Killing vector field in $\R^3$ with respect to the Euclidean metric. As $V$ is a conformal vector field with respect to the Euclidean metric, under the conformal change $\bar{g}=e^{2h}\left\langle ,\right\rangle$, $V$ is also a conformal vector field with respect to the metric $\bar{g}$ and satisfies 
$$
\mathcal{L}_{V}\bar{g}=2(V\left(h\right))\bar{g}.
$$
As
$$ 
 V\left(h(x)\right)=\left\langle \nabla h,V\right\rangle=2u'(\left|\vec{x}\right|^2)\left\langle \vec{x}, V\right\rangle=2u'(\left|\vec{x}\right|^2)\left\langle \vec{x},\vec{x} \wedge E \right\rangle=0, 
$$
we conclude that $V$ is a killing vector field with respect to metric $\bar{g}$. Now, define the function $v:\Sigma \rightarrow \R $ by
\begin{equation}
v(x)=\bar{g}(V,N),
\end{equation} 
where $N$ is a unit normal vector at the point $x\in \Sigma$. By the Proposition 1 in \cite{fornari2004killing}, the function $v$ is a Jacobi function, is that, $v$ is a solution of the problem 
\begin{equation}\label{eqche2}
\Delta_{_\Sigma} v + \left(\bar{R}\mbox{ic}(N) + \left|A\right|^2 \right)v = 0, 
\end{equation}
where $\bar{R}$ic$(N)$ denote the Ricci curvature of $(\B_r^{3},\bar{g})$ in the direction $N$. Notice that at points of $\gamma$ we have
$$
v(\gamma(t))=\bar{g}\left( V,N \right) =\bar{g}\left( \vec{x}\wedge E, N \right)=e^{2h}\left\langle \vec{x}\wedge E, N \right\rangle=0,
$$
since $N$ is parallel to $\vec{x}$ when restricts to $\gamma$. Then,   
\begin{equation}\label{cr1}
\dfrac{d}{dt}(v\circ \gamma )(t)=0\,, \quad \forall \, t.
\end{equation}
For each $t_0 \in [0,1]$, let us consider the curve $\beta:(-\varepsilon,\varepsilon)\rightarrow \Sigma$ such that $\beta(0)=\gamma(t_0)$ and $\beta'(0) \bot \gamma'(t_0)$. Let $N(s)$ be a restriction of the vector field  $N$ to $\beta$. Then,
\begin{eqnarray}\label{eqderivjac}
\frac{d}{ds}(v \circ \beta)(s)_{\displaystyle|_{s=0}}
																				&=& \frac{d}{ds}\left( e^{2h} \right)_{\displaystyle|_{s=0}}\left\langle V,N \right\rangle_{|_{s=0}} \nonumber\\ & & + e^{2h}\left\{ \left\langle \beta'(0)\wedge E,N(s)\right\rangle + \left\langle \beta(0) \wedge E,N'(0)\right\rangle\right\},
\end{eqnarray}
since $\vec{x}// N$ on $\gamma$, $E // \beta'(0)$ and $N'(0) // E$. 
Therefore, the equation $(\ref{eqderivjac})$ becomes 
$$
\frac{d}{ds}(v \circ \beta)(s)_{|_{s=0}} = 0.
$$

This shows that $\gamma(t)$ is a critical point of $v:\Sigma\to\mathbb{R}, \forall t.$
By a result due to S. Y. Cheng (\cite{cheng1976eigenfunctions}, Theorem 2.5), the critical points on nodal set $v^{-1}(0)$ of a function which satisfies an equation like $(\ref{eqche2})$ contains only isolated critical points. As we have seen above, every point in $\gamma\subset v^{-1}(0)$ is a critical point of $v$. In this case, we conclude that $v\equiv 0$ and therefore $\Sigma$ is tangent to $V$. As the vector field $V$ on $\R^3$ is induced by rotations around the axis that is orthogonal to the plane $\pi$, and passes through the center of great circle which contains $\gamma$, we say that $\Sigma$ is rotationally symmetric and the item $ii)$ follows.

\end{proof}

\begin{corollary}\label{corodesiestr} Let $\Sigma$ be a compact free boundary CMC surface in $(\B_r^3,\bar{g})$.  If \, 

$$ \dfrac{|\Phi|^2}{\sigma^2} \bar{g}(\vec{x},N)^2 < \dfrac{1}{2}\left(2+\dfrac{H}{\sigma}\bar{g}(\vec{x},N)\right)^2,$$
then $\Sigma$ is diffeomorphic to a disk $\mathbb{D}^2$.  
\end{corollary}

\begin{proof}
We just note that the condition above ensures that booth eigenvalues of the function $\Psi$ are positive. Therefore, the function $\Psi$ is strictly convex and consequently, the set $\mathcal{C}$
has a single point. 
\end{proof}

\begin{observation} When $(\B_r^3,\bar{g})$ has constant sectional curvature, Nitsche \cite{nitsche1985stationary}, Ros and Souam  \cite{ros1997stability} has proved that if $\Sigma$ is a free boundary CMC disk in $(\B_r^{3},\bar{g})$, then $\Sigma$ is totally umbilical. In particular, if $\Sigma$ is a free boundary  minimal disk, then $\Sigma$ is a totally geodesic. Thus, if the condition $(\ref{gapmin})$ is satisfied in $\Sigma$ and occurs the equality at some point $p\in \Sigma$, we conclude that $\Sigma$ is not a disk. Indeed, otherwise we will have    
$$
0=\frac{1}{\sigma^2}\left|A\right|^2 \bar{g}(\vec{x},N)^2 = 2 \mbox{\,\,at some point } \, p \in \Sigma, 
$$
a contradiction and thus precisely, occurs the condition $ii)$ in Teorema\,\ref{teoprin2h0}.

Now, if $(\B_r^3,\bar{g})$ does not have constant sectional curvature and $\Sigma$ is a free boundary minimal surface in $(\B_r^3,\bar{g})$ such that the condition $(\ref{gapmin})$ is satisfied and occurs equality at some point $p \in \Sigma$, it is not possible guarantee that $\Sigma$ is not a disk. Because we do not know if a analogous of Nitsche's Theorem, or results proved by Ros and Souam, remains valid to such spaces.  
\end{observation}

\section{Examples}\label{sec:examples}

In this section, we show that there are CMC surfaces $\Sigma \subset (\B_r^3,\bar{g})$ with strictly convex boundary which satisfies conditions in Theorem \ref{classifcmc}. The case where $\bar{g}$ is a metric of constant sectional curvature the examples was considered by \cite{ambrozio2016gap} and \cite{HaXa}. 

Notice that if $\Sigma$ is a free boundary totally geodesic disk in $(\B_r^3,\left\langle , \right\rangle)$, the equation   $(\ref{k_i-barra01})$ says that $\Sigma$ is also a free boundary totally geodesic disk in $(\B_r^3,\bar{g})$  and therefore, the condition $(\ref{gapmin})$ in Theorem \ref{classifcmc} is trivially satisfied. Thus, we will focus in the case where the surface $\Sigma$ is rotationally symmetric with nontrivial topology.

Let $x:(-c,c)\rightarrow\R$ $(c>0)$ be a smooth function with $x(t)>0\ ,\forall\, t\in(-c,c)$. Consider the surface $\Sigma \subset (\R^3, \left\langle , \right\rangle )$ obtained by revolution of the curve $\beta(t)=(x(t),0,t)$ around the $Oz-$axis, whose parametrization is given by 
\begin{equation}\label{equaminrevolconf}
X(t,\theta)=(x(t)\cos(\theta),x(t)\sin(\theta),t).
\end{equation}
Consider in $\Sigma$ the orientation:
$$
N=\frac{1}{\sqrt{1+x'(t)^2}}(-\cos(\theta),-\sin(\theta),x'(t)).
$$

To avoid confusion with respect to notation, until the finish of this section, we will use $N$ and $\bar{N}$ to denote the unit  normal vector field to $\Sigma\subset(\mathbb{B}_r^3,\left\langle , \right\rangle)$ and $\Sigma\subset(\B^3_r,\bar{g})$ respectively. Moreover, $x$ will denote the function $x:(-c,c)\rightarrow\R$, whereas $\vec{x}$ will denote the position vector given by $\vec{x}=X(t,\theta)$. 

The principal curvatures of $\Sigma$ at point $(t,\theta)$ are given by 
\begin{equation}\label{k1}
k_1=\frac{-x''(t)}{\sqrt{(1+x'(t)^2)^3}} \mbox{\, and \, }
k_2=\frac{1}{x(t)\sqrt{1+x'(t)^2}}.
\end{equation}

For each $t \in (-c,c)$, consider the curve 
$$
\alpha_t(\omega)=X(\{t\}\times [0,2\pi x(t)])=\left(x(t)\cos\left(\frac{\omega}{x(t)}\right),x(t)\sin\left(\frac{\omega}{x(t)}\right),t\right) 
$$
parametrized by arc length. Let $\nu_t$ be a normal vector field to $\alpha_t$ given by 
$$
\nu_t(\omega)=\frac{-1}{\sqrt{1+x'(t)^2}}\left(x'(t)\cos\left(\frac{\omega}{x(t)}\right),x'(t)\sin\left(\frac{\omega}{x(t)}\right),1 \right).
$$
Define the function
\begin{equation}\label{funcfcurv}
f(t):=f(t,\omega)=\left\langle \alpha_t''(\omega), \nu_t(\omega) \right\rangle=\frac{x'(t)}{x(t)\sqrt{1+x'(t)^2}}
\end{equation}

\begin{observation}\label{funcfrelcurvgeo}
Notice that if $\alpha_t$ is a boundary of the surface $X([t-\epsilon,t] \times [0,2\pi x(t)])$, then the geodesic curvature of $\alpha_t$ is given by $k_{g_e}(t)=f(t)$. If $\alpha_t$ is the boundary of surface $X([t,t+\epsilon]\times[0,2\pi x(t)])$,  in this case we have $k_{g_e}(t)=-f(t)$. 
\end{observation}

As well known, under conformal change $\bar{g}=e^{2h}\left\langle\,,\,\right\rangle$, the principal curvatures  $k_1$ and $k_2$ becomes 
\begin{equation}\label{k_i-barra}
\bar{k}_i=e^{-h}\left(k_i-N(h)\right)\,\mbox{\,for\,} \, i=1,2 ,
\end{equation}
where $N(h)=\left\langle \texttt{grad}\,h, N \right\rangle$. As $h=u(\left|x\right|^2)$, follows that
$$
\texttt{grad}\,h=2u'(\left|\vec{x}\right|^2)\vec{x},
$$
and we can write the equation $(\ref{k_i-barra})$ as
\begin{equation}\label{k_i-barra01}
\bar{k}_i=e^{-h}\left(k_i-2u'(\left|\vec{x}\right|^2)\left\langle \vec{x},N\right\rangle\right)\quad \quad \quad i=1,2.
\end{equation}

If $(\ref{equaminrevolconf})$ is a parametrization of a CMC surface with mean curvature $\bar{H}$ in $(\B_r^3,\overline{g})$, then the above equation provides 
\begin{equation}\label{equaminconf}
\bar{H}=e^{-h}\left(H-4u'(\left|\vec{x}\right|^2)\left\langle \vec{x},N\right\rangle\right).
\end{equation}

By considering the expression of normal vector $N$ and the position vector $\vec{x}$ above, we have  
\begin{equation}\label{suporteminconf}
\left\langle \vec{x},N\right\rangle = - \frac{x(t)-x'(t)t}{\sqrt{1+x'(t)^2}}. 
\end{equation}

Now, using the expressions of $k_1$, $k_2$ and $(\ref{suporteminconf})$, we can rewrite $(\ref{equaminconf})$ as 
\begin{equation}\label{equaminconf001}
\bar{H}=e^h\left(\frac{-x''(t)x(t)+x'^2+1}{x(t)(1+x'(t)^2)^{3/2}} + \dfrac{4u'(|\vec{x}|^2)}{(1+x'(t)^2)^{1/2}}(x(t)-x'(t)t)\right).
\end{equation}

Thus, the problem to find CMC surfaces in $(\B^3_a,\bar{g})$ parametrized by $(\ref{equaminrevolconf})$ is equivalent to problem to find solutions to the differential equation $(\ref{equaminconf001})$ given the initial conditions of interest. For example, if $u\equiv0$ and $\bar{H}\equiv 0$ then $(\B^3_a,\bar{g})$ (for $a=\infty$) is the euclidean space $\R^3$ and a solution to above equation with initial conditions $x(0)=c_0$ and $x'(0)=0$ is given by 
\begin{equation}\label{coshiper}
x(t)=c_0 \cosh\left(\frac{t}{c_0}\right),
\end{equation}
and it is defined for all $t\in \R$. 

In this setting, we focus on case where $\bar{k}_1+\bar{k}_2=0.$ More precisely, we show that there is a minimal surface in $(\B^3_r,\bar{g})$  with strictly convex boundary which satisfies the condition, 
\begin{equation}\label{gap001}
\frac{1}{\sigma^2}\left|A\right|^2 \bar{g}(\vec{x},\bar{N})^2 \leq 2.
\end{equation}

As we saw in demonstration of Proposition \ref{hesnong}, the above condition 
is equivalent to
\begin{equation}\label{autvarmaior0} 
\left(1+\frac{\bar{k}_1}{\sigma}\bar{g}(\vec{x},\bar{N})\right) \geq 0\,\, \mbox{\,and\,}\,\, \left(1+\frac{\bar{k}_2}{\sigma}\bar{g}(\vec{x},\bar{N})\right) \geq 0.
\end{equation}
The two above inequalities are equivalent to 
\begin{equation}\label{condgeralmingapconf}
-1 \leq \frac{\bar{k}_i}{\sigma}\bar{g}(\vec{x},\bar{N})\leq 1 \,\, \mbox{\,for\,}\,\, i=1 \,\, \mbox{\,or\,},\,\, i=2,
\end{equation}
because $\bar{k}_1+\bar{k}_2=0$.
By choosing $i=2$ we have, 
\begin{eqnarray}\label{autvar001}
\bar{k}_2\bar{g}(\vec{x},\bar{N})&=& e^{-h}\left(k_2 -2u'(\left|\vec{x}\right|^2)\left\langle \vec{x},N \right\rangle\right)e^{2h}\left\langle \vec{x},e^{-h}N\right\rangle \nonumber\\  &=& \left(k_2 - 2u'(\left|\vec{x}\right|^2)\left\langle\vec{x},N \right\rangle\right)\left\langle\vec{x},N \right\rangle \nonumber \\
																&=& \left(\frac{1}{x(t)\sqrt{(1+x'(t)^2)}}  - 2u'(\left|\vec{x}\right|^2)\left(-\frac{x(t)-x'(t)t}{\sqrt{1+x'(t)^2}}\right)\right)\times \nonumber \\ & & \left(- \frac{x(t)-x'(t)t}{\sqrt{1+x'(t)^2}}\right)\nonumber  \\
																&=& \left(-\frac{1}{x(t)\left(1+x'(t)^2\right)} - 2u'(\left|\vec{x}\right|^2)\frac{x(t)-x'(t)t}{1+x'(t)^2}\right)\left( x(t)-x'(t)t \right)\nonumber \\
                                &=& - \frac{\Big(1+2u'(\left|\vec{x}\right|^2)\big(x(t)-x'(t)t\big)x(t)\Big)\big(x(t)-x'(t)t\big)}{x(t)(1+x'(t)^2)}. 
\end{eqnarray}
As $\sigma=1+2u'(\left|\vec{x}\right|^2)\left|\vec{x}\right|^2=1+2u'(\left|\vec{x}\right|^2)(x(t)^2+t^2)$ we have  
$$
\frac{\bar{k}_2}{\sigma}\bar{g}(\vec{x},\bar{N})=- \frac{\Big(1+2u'(\left|\vec{x}\right|^2)\big(x(t)-x'(t)t\big)x(t)\Big)\big(x(t)-x'(t)t\big)}{(1+2u'(x(t)^2+t^2)(x(t)^2+t^2))x(t)(1+x'(t)^2)},
$$
and the condition $(\ref{condgeralmingapconf})$ becomes equivalent to 
\begin{equation}\label{condgeralmingapconf00}
-1 \leq \frac{\Big(1+2u'(\left|\vec{x}\right|^2)\big(x(t)-x'(t)t\big)x(t)\Big)\big(x(t)-x'(t)t\big)}{(1+2u'(\left|\vec{x}\right|^2)(x(t)^2+t^2))x(t)(1+x'(t)^2)} \leq 1.
\end{equation}

Let $\Sigma \subset (\B_{r}^3,\bar{g})$ be a minimal surface parametrized by $(\ref{equaminrevolconf})$, where
$x:[-s,s] \rightarrow \R$ is a solution of $(\ref{equaminconf001})$ to $\bar{H}=0$ and $r=\sqrt{x(s)^2+s^2}$. Thus, it is sufficient to verify the inequality $(\ref{condgeralmingapconf00})$ to the function $x:[-s,s] \rightarrow \R$ so that the condition $(\ref{gap001})$    
be satisfied in $\Sigma$.


When $\bar{g}$ is the Euclidean metric, we have an explicit solution to equation $(\ref{equaminconf001})$ as in $(\ref{coshiper})$. However, the general case $\bar{g}=e^{2u(\left|x\right|^2)}\left\langle\,,\,\right\rangle$ where it is considered any smooth function $u$ we do not have, in general, a explicit solution for the equation $(\ref{equaminconf001})$. Thus, the verification of the condition $(\ref{condgeralmingapconf00})$ to this cases is not a straightforward calculation. 

Let $(\R^3,\bar{g})$ be a Gaussian Space, i.e, the Euclidean space with conformal metric $\bar{g}=e^{2u(\left|x\right|^2)}\left\langle\,,\,\right\rangle$, where $u(\left|x\right|^2)=-\frac{\left|\vec{x}\right|^2}{8}$. Now, we will focus only in this case.


As we saw earlier, the vector field $\vec{x}$ is conformal with respect to $\bar{g}=e^{-\frac{\left|\vec{x}\right|^2}{4}}\left\langle ,\right\rangle$ and satisfies $\mathcal{L}_{\vec{x}}\bar{g}=2\sigma\bar{g}$, where $\sigma=\frac{4-\left|\vec{x}\right|^2}{4}$.

For the immersion $\Sigma \subset (\B^3_r,\bar{g})$, the existence of the function $\Psi:\B^3_r \rightarrow \R$ such that Hess$_{\Sigma}\Psi \geq 0$ is guaranteed by Theorem \ref{classifcmc} provided that the potential function $\sigma$ of conformal vector field $\vec{x}$ satisfies $\sigma>0$. Thus, for $\sigma=\frac{4-\left|\vec{x}\right|^2}{4}$ we have $\sigma > 0$ when $r<2$.

Now, for $u(\left|x\right|^2)=-\frac{\left|\vec{x}\right|^2}{8}$ the equation $(\ref{equaminconf001})$ becomes
\begin{equation}\label{equaminconf00001}
\frac{x''(t)}{1+x'(t)^2}=\frac{1}{x(t)} - \frac{1}{2} (x(t)-x'(t)t).
\end{equation}

\begin{lemma} Let $x:(-c,c) \rightarrow \R$ be a solution to the above equation with initial conditions $x(0)<\sqrt{4-2\sqrt{2}}$ and $x'(0)=0$. Then, there is $0<\delta<c$ and $r=r(\delta)$ such that the parametrization given by $(\ref{equaminrevolconf})$ with $x_{|_{[-\delta,\delta]}}$ generates a minimal surface $\Sigma_{\delta} \subset \left(\B^3_r,e^{-\frac{\left|\vec{x}\right|^2}{4}}\left\langle , \right\rangle \right)$ whose boundary $\partial \Sigma_{\delta} \subset \partial \B_r^3$ is strictly convex. 
\end{lemma}

\begin{proof} Firstly, we observe that 
\begin{equation}\label{solupar}
x(-t)=x(t)\,\forall \,t\in(-c,c).
\end{equation}
In fact, we define $\tilde{x}:(-c,0]\rightarrow\R$ by $\tilde{x}(s)=x(-s)$. Then, $\tilde{x}'(s)=-x'(-s)$ and $\tilde{x}''(s)=x''(-s).$ This implies that
\begin{eqnarray*}
\frac{\tilde{x}''(s)}{1+\tilde{x}'(s)^2}&=&\frac{x''(-s)}{1+x''(-s)^2}\\
																				&=&\frac{1}{x(-s)} - \frac{1}{2} (x(-s)-x'(-s)(-s))\\
                                        &=&\frac{1}{\tilde{x}(s)} - \frac{1}{2} (\tilde{x}(s)-\tilde{x}'(s)s).
\end{eqnarray*}

Therefore, $\tilde{x}$ is also a solution to equation $(\ref{equaminconf00001})$ and besides that, satisfies $\tilde{x}(0)=x(0)$ and $\tilde{x}'(0)=x'(0)=0$. By uniqueness of solution to differential equation we have that $\tilde{x}(s)=x(s)$ in $(-c,0]$, whence the statement follows.  

Thus, the function $f$ defined in $(\ref{funcfcurv})$ satisfies $f(t)=-f(-t)$ and 
$$
\bar{f}(t):=e^{^{\left(\frac{x(t)^2+t^2}{8}\right)}}\left( f(t) - \frac{x(t)x'(t)+t}{4\sqrt{1+x'(t)^2}} \right),\quad t \in (-c,c)
$$
satisfies $\bar{f}(t)=-\bar{f}(-t)$. Fix $0<s<c$ and consider the surface $\Sigma_s$ parametrized by $(\ref{equaminrevolconf})$, where $x$ is  restrict to the interval $[-s,s]$. In this case we have $\partial \Sigma = \alpha_{-s}\cup \alpha_s$. By Observation \ref{funcfrelcurvgeo}, the geodesic curvatures of the boundary of $\Sigma_s$ with respect to the canonical metric of $\R^3$ satisfies, $k_{g_e}(-s)=-f(-s)=f(s)=k_{g_e}(s)$. Now, consider $\Sigma_{s} \subset (\R^3,e^{-\frac{\left|\vec{x}\right|^2}{4}}\left\langle, \right\rangle)$. By equation $(\ref{curvgeomudconf})$, the geodesic curvature of the boundary $\alpha_s \subset \partial \Sigma_s$ is given by $\bar{k}_{g_e}=e^{\frac{\left|\vec{x}\right|^2}{8}}\left( k_{g_e} - \nu(h) \right)$. Thus,  
\begin{eqnarray*}
\bar{k}_{g_e}(s)&=&e^{-h} \left( k_{g_e} - \nu(h) \right) \\
                &=& e^{^{\left(\frac{x(s)^2+s^2}{8}\right)}}\left( \frac{x'(s)}{x(s)\sqrt{1+x'(s)^2}} - \frac{x(s)x'(s)+s}{4\sqrt{1+x'(s)^2}} \right) \\
								&=& \bar{f}(s),
\end{eqnarray*}
and in analogous way, we have that $\bar{k}_{g_e}(-s)=-\bar{f}(-s)$. Observe that $\bar{k}_{g_e}(0)=0$, i.e, the curve $X(\{0\}\times[0,2\pi])$ is a geodesic in $\Sigma_{s} \subset (\R^3,e^{-\frac{\left|\vec{x}\right|^2}{4}}\left\langle, \right\rangle)$.






As $\bar{f}(s)=-\bar{f}(-s)$, we have that $\bar{f}$ is a odd function. We rewrite 
$$
\bar{f}(t)=a(t)b(t),
$$
where
\begin{eqnarray*}
 a(t)=\frac{e^{^{\left(\frac{x(t)^2+t^2}{8}\right)}}}{\sqrt{1+x'(t)^2}}\ \text{and}\ b(t)= \left( \frac{x'(t)}{x(t)} - \frac{x(t)x'(t)+t}{4} \right).
\end{eqnarray*}

Thus, 
\begin{eqnarray*}
\frac{d\bar{f}}{dt}(0)
											= e^{^{\left(\frac{x(0)^2}{8}\right)}}\left(  \frac{x(0)^4-8x(0)^2+8}{8x(0)^2} \right).
\end{eqnarray*}  

We observe that if $0<x(0)<\sqrt{4-2\sqrt{2}}$ then $x(0)^4-8x(0)^2+8>0$. This implies $\frac{d\bar{f}}{dt}(0)>0$ and thus, $\bar{f}$ is increasing in a neighbourhood of $0$. Therefore, there is $\delta>0$ such that $\bar{f}(\delta)>0$ and the surface $\Sigma_{\delta}$ as described above satisfies $\bar{k}_{g_e}(-\delta)=\bar{k}_{g_e}(\delta)=\bar{f}(\delta)>0$. To $r=\sqrt{x(\delta)^2+\delta^2}$ we have that $\Sigma_{\delta} \subset \B_r^3$ and moreover, $\partial \Sigma_{\delta}\subset \partial \B_r^3$. Thus, $\Sigma_{\delta} \subset \left(\B^3_r,e^{-\frac{\left|\vec{x}\right|^2}{4}}\left\langle , \right\rangle \right)$ has strictly convex boundary $\partial \Sigma_{\delta}$  
\end{proof}

\begin{lemma}\label{condinici} Let $x:(-c,c) \rightarrow \R$ be a solution to equation $(\ref{equaminconf00001})$ with initial conditions $x'(0)=0$ and $x(0)=x_0<\sqrt{4-2\sqrt{2}}$. 
Then, there is $\varepsilon>0$ such that the condition $(\ref{condgeralmingapconf00})$ is satisfied for all $t\in(-\varepsilon,\varepsilon).$ 
\end{lemma}

\begin{proof}

Define the function 
\begin{equation}\label{gapequation001}
\mathcal{F}(t):=\frac{\left[4-x(t)(x(t)-x'(t)t)\right](x(t)-x'(t)t)}{x(t)(4-x(t)^2-t^2)(1+x'(t)^2)}.
\end{equation}

It is straightforward to check that $\mathcal{F}(0)=1$, $\mathcal{F}'(0)=0$ and $\mathcal{F}''(0)<0$.  
Thus, the function $\mathcal{F}$ admits a local maximum at $0\in (-c,c)$. By continuity, there is $\varepsilon>0$ such that $\mathcal{F}(t)\leq 1\,\forall\, t\in(-\varepsilon,\varepsilon)$.
We can to simplify the calculations in order to show that $\mathcal{F}''(0)<0.$ In fact, we writing $\mathcal{F}(t)=a(t)b(t)c(t),$ where
$$a(t)= \frac{4}{x(t)}-b(t), \ b(t)= x(t)-x'(t)t \ \text{and} \ c(t)= \dfrac{1}{(4-x(t)^2-t^2)(1+x'(t)^2)}$$
and use the fact that $x''(0)=\frac{2-x(0)^2}{2x(0)}$ to obtain
\begin{equation*} 
\mathcal{F}''(0)=\frac{-x(0)^4+8x(0)^2-8}{2x(0)^2}. 
\end{equation*}

As in the previous proposition, we have $-x(0)^4+8x(0)^2-8<0$ to 
$0<x(0)<\sqrt{4-2\sqrt{2}}$ and consequently $\mathcal{F}''(0)<0$,  as desired.  
\end{proof}
In consequence of two previous lemmas we have:

\begin{theorem} There is $r>0$ and a minimal surface $\Sigma \subset (\B_r^3,e^{-\frac{\left|\vec{x}\right|^2}{4}}\left\langle, \right\rangle)$ with strictly convex boundary $\partial \Sigma \subset \partial\B^3_r$ where the condition 
$$\frac{1}{\sigma^2}\left|A\right|^2 \bar{g}(\vec{x},\bar{N})^2 \leq 2 
$$
is satisfied for all point $p$ in $\Sigma$.
\end{theorem}

\begin{proof}
Let $x:[-\xi,\xi]\rightarrow \R$ be a solution to equation  $(\ref{equaminconf00001})$ with initial conditions $x'(0)=0$ and $0<x(0)<\sqrt{4-2\sqrt{2}}$, where $\xi=\min\{\delta,\varepsilon\}$ being $\delta$ and $\varepsilon$ given by two previous lemma. Thus, we consider $\Sigma_{\xi}\subset (\B_r^3,e^{-\frac{\left|\vec{x}\right|^2}{4}}\left\langle, \right\rangle)$, where $r=\sqrt{x(\xi)^2+\xi^2}$ and the theorem follows.
\end{proof}

The above Theorem shows the existence of some surfaces under hypothesis of Theorem $(\ref{classifcmc})$. However, the next example a  little bit more interesting.

Singurd B. Angenet shows in \cite{Ang} the existence of a non-circular minimal torus $\mathbb{T}_{{\mathcal{A}}} \subset  \left(\R^3,e^{\frac{-\left|\vec{x}\right|^2}{4}}\left\langle , \right\rangle\right)$ obtained by revolution of planar closed simple curve around a fixed axis. Let us assume that such curve is in $xz-$plane and it is being rotated around the $z-$axis (see Figure 1). 
\begin{center}
\begin{tikzpicture}


\draw [->] (-5.2,0) -- (5.2,0);
\draw [->] (0,0) -- (0,8.5);
\draw (-4.5,0) .. controls (-4.5, 2.5) and (-2.5, 4.5) .. (0, 4.5) 
               .. controls (2.5,4.5) and (4.5,2.5) .. (4.5,0);

\draw[rounded corners=1.5cm] (-1.5,0) -- (-1.5,1.5) --(0,1.5);

\draw[rounded corners=1.5cm] (0,1.5) -- (1.5,1.5) --(1.5,0);

\draw[rounded corners=1.5cm] (-1.5,0) -- (-1.5,-1.5) --(0,-1.5);

\draw[rounded corners=1.5cm] (0,-1.5) -- (1.5,-1.5) --(1.5,0);

\draw[blue] (-0.74,1.3) .. controls (-0.92,0.35) and (-0.92,-0.35).. (-0.74,-1.3);

\draw[blue] (0.74,1.3) .. controls (0.55,0.35) and (0.55,-0.35) .. (0.74,-1.3);

\draw[blue] (-0.74,-1.3) .. controls (-0.25,-0.85) and (0.25,-0.85) .. (0.74,-1.3);

\draw[blue,dashed] (-0.74,1.3) .. controls (-0.5,0.4) and (-0.5,-0.35).. (-0.74,-1.3);

\draw[blue,dashed] (0.74,1.3) .. controls (0.93,0.4) and (0.93,-0.35).. (0.74,-1.3);

\draw (-1.5,0) .. controls (-1.7,-0.9) and (1.4,-0.7).. (1.5,0);

\draw[dashed] (-1.5,0) .. controls (-1.2,0.9) and (1.8,0.7).. (1.5,0);


\draw [blue] (0,0.95) .. controls (0.65,0.95) and (1.7,1.9) .. (1.85,4.6); 

\draw [blue] (0,0.95) .. controls (-0.65,0.95) and (-1.7,1.9) .. (-1.85,4.6);

\draw [blue] (1.85,4.6) .. controls (1.85,7) and (1,7.9) .. (0,7.9); 

\draw [blue] (-1.85,4.6) .. controls (-1.85,7) and (-1,7.9) .. (0,7.9);


\node (3) at (5.2,-0.2) {\small {$z$}};
\node (4) at (0.2,8.3) {\small{$x$}};
\node (5) at (-0.2,4.7) {\small{$2$}};
\node (6) at (-0.81,2.82) {\footnotesize {$\sqrt{4-2\sqrt{2}}$}};
\node (7) at (3.6,2.0) {\small{$ \mathbb{S}^{2}_{2}$}};
\node (8) at (1.65,7) {\small {$ \mathbb{T}_{\mathcal{_A}}$}};
\node (9) at (0,0.95){\small {$ - $}};

\node (10) at (-0.2,0.82){\small {$ x_1 $}};
\node (11) at (0,2.82){\small {$ - $}};
\node (12) at (1.7,-0.2){\small {$ r $}};
\node (13) at (1.7,0.8){\small {$\mathbb{B}^3_r $}};
\end{tikzpicture}

\begin{flushleft}
{\small Figure 1. The surface in the previous example is obtained by intersection of Angenent 
torus with the ball $\B^3_r$ in the space $\big(\R^3,\bar{g} \big)$, where  $\bar{g}=e^{-\frac{\left|\vec{x}\right|^2}{4}}\left\langle , \right\rangle$ .} 

\end{flushleft}

\end{center}

Angenet shows that such curve is symmetric with respect to $x-$axis intercepting it orthogonally at two points $0<x_1<x_2$. In \cite{moller2011closed}, Proposition 2.1, M\"oller got more information about the localization of $\mathbb{T}_{{\mathcal{A}}}$. Among them, one of interest us is about the value of $x_1$, more precisely, was obtained that
$$
\frac{7}{16}-\frac{3}{98}<x_1<\frac{7}{16} + \frac{3}{98}.
$$



Let $\beta(t)=(x(t),0,t)$ be a local parametrization of the curve which generates $\mathbb{T}_{{\mathcal{A}}}$ such that $x:(-c,c) \rightarrow \R$ satisfies $x(0)=x_1$ and $x'(0)=0$. We have $x(0)=x_1<\frac{7}{16} + \frac{3}{98} \approx 0,46811... < \sqrt{4-2\sqrt{2}} \approx 1,08239.$ The previous Theorem provides de following example. 

\begin{example} There is $r=r(x_1)<2$ such that $\Sigma := \mathbb{T}_{{\mathcal{A}}} \cap \B^3_r$ is a immersed minimal surface $\Sigma \subset \left(\B^3_r,e^{\frac{-\left|\vec{x}\right|^2}{4}}\left\langle , \right\rangle\right)$ where the gap condition 
\begin{equation}\label{Gapexam}
\frac{1}{\sigma^2}\left|A\right|^2 \bar{g}(\vec{x},\bar{N})^2 \leq 2
\end{equation}
is satisfied for all points $x$ in $\Sigma$ and whose boundary $\partial \Sigma \subset \B^3_r$ is strictly convex. 
\end{example}

\vspace{0.2cm}

We expect that $r$ can be choosed in such way that $r<2$ and $\Sigma=\mathbb{T}_{{\mathcal{A}}} \cap \B^3_r$ is a free boundary minimal surface in $\left(\B^3_r,e^{-\frac{\left|\vec{x}\right|^2}{4}}\left\langle , \right\rangle \right)$ where the gap condition (\ref{Gapexam}) is satisfied. And if that is the case, an interesting question is to know if for this $r$ fixed, the totally geodesic disk and the surface $\Sigma=\mathbb{T}_{{\mathcal{A}}} \cap \B^3_r$ are only surfaces in 
$\left(\B^3_r,e^{-\frac{\left|\vec{x}\right|^2}{4}}\left\langle , \right\rangle\right)$ satisfying the condition $(\ref{Gapexam})$.

\bibliography{main.bib}
\bibliographystyle{acm}

\end{document}